\documentclass[reqno, 12pt]{amsart}
\usepackage{a4wide, amsfonts, amssymb, amsmath, amsthm, mathrsfs, enumerate, enumitem, setspace, url}
\usepackage[utf8]{inputenc}
\usepackage{graphicx}
\usepackage{tikz-cd} 
	\usetikzlibrary{cd}
\usepackage[all]{xy}
\usepackage{mathtools}
\usepackage{lmodern}
\usepackage[french,english]{babel}

\usepackage{color}
\usepackage{enumitem}
\usepackage{algpseudocode}
\usepackage{hyperref} 
	\definecolor{darkred}{rgb}{0.5,0,0}
	\definecolor{darkgreen}{rgb}{0,0.5,0}
	\definecolor{darkblue}{rgb}{0,0,0.5}
	\hypersetup{linkcolor=darkblue,filecolor=darkgreen,urlcolor=darkred,citecolor=darkblue}

\usepackage{comment}
\usepackage{longtable}
\usepackage[T2A,T1]{fontenc}
\DeclareSymbolFont{cyrillic}{T2A}{cmr}{m}{n}
\DeclareMathSymbol{\Sha}{\mathalpha}{cyrillic}{216}

\usepackage[toc,page]{appendix}

\theoremstyle{plain}
\newtheorem{theorem}{Theorem}[section]
\newtheorem*{theorem*}{Theorem}
\newtheorem{proposition}[theorem]{Proposition}
\newtheorem{lemma}[theorem]{Lemma}
\newtheorem{question}[theorem]{Question}

\theoremstyle{remark}
\newtheorem{remark}[theorem]{Remark}

\newtheorem*{acknowledgements}{Acknowledgements}

\theoremstyle{definition}
\newtheorem{definition}[theorem]{Definition}
\newtheorem{condition}[theorem]{Condition (D)}
\newtheorem{conjecture}[theorem]{Conjecture}
\newtheorem{notation}[theorem]{Notation}
\newtheorem{assumption}[theorem]{Assumption}
\numberwithin{equation}{section}

\DeclareMathOperator{\HH}{H}

\DeclareMathOperator{\Pic}{Pic}
\DeclareMathOperator{\Br}{Br}
\DeclareMathOperator{\inv}{inv}
\DeclareMathOperator{\ev}{ev}

\DeclareMathOperator{\Spec}{Spec}

\DeclareMathOperator{\im}{Im}

\DeclareMathOperator{\Sel}{Sel}

\DeclareMathOperator{\spann}{Span}

\DeclareMathOperator{\valtwo}{val_{\it 2}}

\DeclareMathOperator{\valv}{val_{\it v}}
\DeclareMathOperator{\valx}{val_{\it x}}
\DeclareMathOperator{\valuiz}{val_{\it u_{i}^{0}}}
\DeclareMathOperator{\valu}{val_{\it u_{i}}}

\DeclareMathOperator{\Hom}{Hom}
\DeclareMathOperator{\et}{\acute{e}t}
\renewcommand{\epsilon}{\varepsilon}

\newenvironment{poliabstract}[1]
  {\renewcommand{\abstractname}{#1}\begin{abstract}}
  {\end{abstract}}
\begin{document}
\onehalfspacing
\title[Integral points on affine surfaces fibered over $\mathbb{A}^{1}$]{Integral points on affine surfaces fibered over $\mathbb{A}^{1}$}

\author{H.Uppal}
	\address{}
	\email{}

\date{\today}
\maketitle
\selectlanguage{french}
\begin{poliabstract}{Résumé}
   Profitant d'un travail précédent d'Harpaz nous utilisons la méthode de descente-fibrations de Swinnerton-Dyer pour étudier les points entiers sur des surfaces affines qui sont des fibrations de tores de norme 1 sur $\mathbb{A}^{1}$.
   \end{poliabstract}
\selectlanguage{english}
\begin{poliabstract}{Abstract}
Taking advantage of previous work of Harpaz we use Swinnerton-Dyer's descent-fibration method to study integral points on affine surfaces which are fibrations of norm 1 tori over $\mathbb{A}^{1}$.
\end{poliabstract}

\setcounter{tocdepth}{1}
\tableofcontents
\section{Introduction}
In \cite{SD95} Swinnerton-Dyer pioneered the descent-fibration method which consists of combining the fibration method with descent on genus 1 curves. Later Colliot-Thélène, Skorobogatov and Swinnerton-Dyer \cite{CTSSD98}, Skorobogtov and Swinnerton-Dyer \cite{SSD05}, Wittenberg \cite{O07}, Harpaz and Skorobogtov \cite{HS16} adapted the method further to a wider range of settings. The idea behind the method is: given a genus 1 fibration $f:X\rightarrow Y$, find a rational point on the base where the fibre above this point is everywhere locally soluble and such that a suitable Selmer group associated to this fibre is small enough forcing this fibre to have a rational point. The adaption we are interested in for this article is one of Harpaz \cite{H19}, where he uses Swinneton-Dyer's method on a fibration of norm 1 tori over the projective line to study integral points on conic log $K$3 surfaces. In this article, similarly to Harpaz we will study a fibration of norm 1 tori; however the contrast between our adaptations are twofold \begin{enumerate}
\item In this article we have better control of certain "Selmer" and "dual Selmer" groups, compared to that of \cite{H19}.
\item In \cite{H19}, Harpaz is able gain unconditional results. However, in this article our results are conditional on the Schinzel Hypothesis (H) (see Section \ref{sec: Schinzel's Hypothesis}).
\end{enumerate} 
\begin{theorem}\label{thm: main thm}
Let $S_{0}$ be a finite set of places of $\mathbb{Q}$ containing $\infty$. Let $J$ be a non empty set and for each $i\in J$, let $c_{i},d_{i}\in \mathbb{Z}_{S_{0}}$ be coprime elements such that $\Delta_{i,j} :=c_{i}\frac{d_{j}}{c_{j}}-d_{i}\neq 0$ for $i\neq j$. Denote by $p_{i}(t)=c_{i}t+d_{i}$ and $p_{J'}=\prod\limits_{i\in J'}p_{i}$ for $J'\subseteq J$. For $a,b\in \mathbb{Z}_{S_{0}}\backslash \{0\}$ and a partition $J=A\cup B$, consider the scheme \[
\mathcal{U}:ap_{A}(t)x^2+bp_{B}(t)y^2=1\subseteq \mathbb{A}^{3}_{\mathbb{Z}_{S_{0}}}
\]over $\mathbb{Z}_{S_{0}}$. Let $\pi:\mathcal{U}\rightarrow \mathbb{A}^{1}, (x,y,t)\mapsto t$. Under the assumption of Schinzel's Hypothesis (H), if there exists an $S_{0}$-integral adelic point $(P_{v})=(x_{v},y_{v},t_{v})\in \mathcal{U}(\mathbb{A}_{S_{0}})$ such that
\begin{enumerate}
\item Condition (D) (See Section \ref{sec: condition (D)}) holds
\item $\valv(dp_{J}(t_{v}))\leq 1$ for all $v\not\in S_{0}$ and $\valtwo(dp_{J}(t_{v}))=1$ if $2\not\in S_{0}$
\item There exists a place $v\in S_{0}$ for which $-dp_{J}(t_{v})$ is a non-zero square at $\mathbb{Q}_{v}$
\item $P_{v}\in \mathcal{U}(\mathbb{A}_{\mathbb{Z}_{S_{0}}})^{\Br^{\text{vert}}(X,
\pi)}$
\end{enumerate} then $\mathcal{U}(\mathbb{Z}_{S_{0}})\neq \emptyset$.
\end{theorem}

\begin{acknowledgements} This research took place while the author was visiting Université Sorbonne Paris Nord. He would like to thank all the members of the university for their hospitality and for providing ideal working conditions. He would also like to thank Yonatan Harpaz for suggesting the problem and for taking time to answer all of his questions. The author is indebted to Olivier Wittenberg for his immeasurable support, insights and explanations about the details behind the descent-fibration method; this research owes much to the conversations held with him.
\end{acknowledgements}

\begin{notation}
Given a number field $k$ we denote by $\Omega_{k}$ for the set of places of $k$. By a variety $V$ over a field $k$ we mean a scheme of finite type over $k$ which is smooth, separated and geometrically integral. Given an integral scheme $S$ of a field $k$ we denote by $k(S)$ for the function field of $S$.
\end{notation}
\section{Norm one tori}\label{sec: norm one tori} We give a brief recap of norm 1 tori, if the reader wants a more detailed exposition we refer them to \cite[\S 2.2]{H19}. For the rest of Section \ref{sec: norm one tori} fix a number field $k$, a finite set of places $S_{0}$ containing all archimedian places and let $d\in \mathcal{O}_{S_{0}}$ be a non-zero $S_{0}$-integer. Denote by $\mathcal{T}_{d}$ the $\mathcal{O}_{S_{0}}$-scheme \[ 
\mathcal{T}_{d}:x_{0}^2-dx_{1}^2=1.
\] For every $a\mid d$ we denote by $\mathcal{U}_{a,b}$ the $\mathcal{O}_{S_{0}}$-scheme \[
\mathcal{U}_{a,b}:ax_{0}^2+bx_{1}^2=1
\] where $b=-\frac{d}{a}$ and $(a,b)=1$ in $\mathcal{O}_{S_{0}}$. The action of $\mathcal{T}_{d}$ on $\mathcal{U}_{a,b}$ exhibits $\mathcal{U}_{a,b}$ as a torsor under $\mathcal{T}_{d}$. Let $K:=k(\sqrt{d})$ and $S:=S_{0}\cup \{\text{all places of $k$ which ramify in $K$}\}$. Further, let $T\subseteq \Omega_{K}$ be the set of places of $K$ lying above $S$.

\subsection{Selmer and Tate Shafarevich groups}\label{subsec: selmer and sha} Throughout Section \ref{sec: norm one tori}, let $\mathcal{T}:=\mathcal{T}_{d}\times_{\Spec \mathcal{O}_{S_{0}}} \Spec \mathcal{O}_{S}$ and $\hat{\mathcal{T}}:=\Hom_{\text{Sh}_{\et}}(\mathcal{T},\mathbb{G}_{m})$ where $\text{Sh}_{\et}$ denotes the category of étale sheaves.
\begin{definition}
Denote by $\Sha^{1}(\mathcal{T},S)$ the kernel\[
\Sha^{1}(\mathcal{T},S):=\ker\left(\HH^{1}_{\et}(\mathcal{O}_{S},\mathcal{T})\rightarrow \prod\limits_{v\in S}\HH^{1}_{\et}(k_{v},\mathcal{T}\otimes_{\mathcal{O}_{S}}k_{v})\right).
\]
\end{definition}
\begin{remark}
Given a class $[T]\in \HH^{1}_{\et}(\mathcal{O}_{S},\mathcal{T})$ we always have $T(\mathcal{O}_{v})\neq \emptyset$ for $v\not\in S$ i.e. the image of the map \[
\HH^{1}_{\et}(\mathcal{O}_{S},\mathcal{T})\rightarrow \HH^{1}_{\et}(k_{v},\mathcal{T}\otimes_{\mathcal{O}_{S}} k_{v})
\] is zero. This is beacuse the map $\HH^{1}_{\et}(\mathcal{O}_{v},\mathcal{T})\rightarrow \HH^{1}_{\et}(k_{v},\mathcal{T}\otimes_{\mathcal{O}_{S}}k_{v})$ factors through $f:\HH^{1}_{\et}(\mathcal{O}_{v},\mathcal{T})\rightarrow \HH^{1}(\mathbb{F}_{v},\mathcal{T}')$ where $\mathcal{T}'$ is the special fibre of $\mathcal{T}\rightarrow \Spec \mathcal{O}_{v}$. The map $f$ is an isomorphism and by Lang's Theorem \cite[Thm 4.4.17]{S98} $\HH^{1}(\mathbb{F}_{v},\mathcal{T}')=0$.
\end{remark}
\begin{definition}
Denote by $\Sha^{2}(\hat{\mathcal{T}},S)$ the kernel\[
\Sha^{2}(\hat{\mathcal{T}},S):=\ker\left(\HH^{2}_{\et}(\mathcal{O}_{S},\hat{\mathcal{T}})\rightarrow \prod\limits_{v\in S}\HH^{2}_{\et}(k_{v},\hat{\mathcal{T}}\otimes_{\mathcal{O}_{S}}k_{v})\right).
\]
\end{definition}
\begin{remark}  There exists: \begin{enumerate}
\item A perfect pairing $
\Sha^{1}(\mathcal{T},S)\times \Sha^{2}(\hat{\mathcal{T}},S)\rightarrow \mathbb{Q}/\mathbb{Z}
$.
\item The short exact sequence of étale sheaves \[
0\rightarrow \mathbb{Z}/2\mathbb{Z}\rightarrow \mathcal{T} \xrightarrow[]{\times 2}\mathcal{T}\rightarrow 0,
\] and the corresponding sequence in étale cohomology give rise to a map $\phi_{1}:\HH^{1}_{\et}(\mathcal{O}_{S},\mathcal{T})\rightarrow \HH^{1}_{\et}(\mathcal{O}_{S},\mathcal{T})$.
\item The short exact sequence of étale sheaves \[
0\rightarrow \hat{\mathcal{T}}\xrightarrow[]{2} \hat{\mathcal{T}} \rightarrow \mathbb{Z}/2\mathbb{Z}\rightarrow 0,
\] and the corresponding sequence in étale cohomology give rise to a map $\phi_{2}:\HH^{1}_{\et}(\mathcal{O}_{S},\mathbb{Z}/2\mathbb{Z})\rightarrow \HH^{2}_{\et}(\mathcal{O}_{S},\hat{\mathcal{T}})$.
\end{enumerate}
\end{remark}
\begin{definition}
We define the \emph{Selmer group} of $\mathcal{T}$ over $\mathcal{O}_{S}$ to be\[
\Sel(\mathcal{T},S):=\{x\in \HH^{1}(\mathcal{O}_{S},\mathcal{T}):\phi_{1}(x)\in \Sha^{1}(\mathcal{T},S)\},
\] and the \emph{dual Selmer group} of $\mathcal{T}$ over $\mathcal{O}_{S}$ to be \[
\Sel(\hat{\mathcal{T}},S):=\{x\in \HH^{1}(\mathcal{O}_{S},\mathbb{Z}/2\mathbb{Z}):\phi_{2}(x)\in \Sha^{2}(\hat{\mathcal{T}},S)\}.
\]
\end{definition}

\subsection{Duality and its consequences}\label{sec: duality and its consequences}
Throughout Section \ref{sec: duality and its consequences}, let $S'$ be any finite set of places containing $S$ such that $\Pic \mathcal{O}_{S'}=0$.
\begin{definition} For a $v\in S'$ denote by $V^{v}$,$V_{v}$ \[
V^{v},V_{v} := \HH^{1}_{\et}(k_{v},\mathbb{Z}/2\mathbb{Z})\cong k_{v}^{*}/(k_{v}^{*})^2.
\]
\end{definition}
\begin{definition} We define $I^{S'}$ and $I_{S'}$ to be \[
I^{S'},I_{S'}:=\mathcal{O}_{S'}^{*}/(O^{*}_{S'})^2,
\]
\end{definition}
\begin{definition}\label{defn: Hilbert symbol pairing}
We define a non-degenerate pairing \[
\langle - , - \rangle_{v}: V_{v}\times V^{v}\rightarrow \mathbb{Z}/2\mathbb{Z}, \ \ \ (a,b)\mapsto \langle a,b\rangle_{v}.
\]  where $\langle -,-\rangle_{v}$ is the Hilbert symbol at the place $v$. Moreover, we can extend this pairing to a non-degenerate pairing 
\[
\langle - , - \rangle_{S'}: V_{S'}\times V^{S'}\rightarrow \mathbb{Z}/2\mathbb{Z}, \ \ \ \left((v_{1},\dotsc,v_{|S'|}),(v'_{1},\dotsc,v'_{|S'|})\right)\mapsto \sum\limits_{v\in S'}\langle v_{i},v'_{i}\rangle_{v}.
\]
\end{definition}
\begin{remark} \leavevmode
\begin{enumerate}
\item As $S'$ contains all the real places, all places above 2 and $\Pic \mathcal{O}_{S'}=0$, we have that $I_{S'},I^{S'}\cong \HH^{1}_{\et}(\mathcal{O}_{S'},\mathbb{Z}/2\mathbb{Z})$ and localisation maps \begin{align*}
I_{S'}:=\HH^{1}_{et}(\mathcal{O}_{S'},\mathbb{Z}/2\mathbb{Z})\hookrightarrow \HH^{1}_{\et}(k_{v},\mathbb{Z}/2\mathbb{Z})=:V_{S'}\\
I^{S'}:=\HH^{1}_{et}(\mathcal{O}_{S'},\mathbb{Z}/2\mathbb{Z})\hookrightarrow \HH^{1}_{\et}(k_{v},\mathbb{Z}/2\mathbb{Z})=:V^{S'}
\end{align*}
\item $I^{S'}$ is the orthogonal complement of $I_{S'}$ via $\langle -,-\rangle_{S'}$ and vise versa.
\end{enumerate}
\end{remark}
\begin{definition}\label{defn: definition of Wv}
For each $v\in S'$ we define the subspace $W^{v}\subseteq V^{v}$ to be
\[ W^{v}:=\begin{cases} [d] \ &\text{if}\ v\in S,\\
\im\left(\mathcal{O}_{v}^{*}/(\mathcal{O}_{v}^{*})^{2}\rightarrow k_{v}^{*}/(k_{v}^{*})^2\right)\ &\text{if}\ v\in S'\backslash S.
\end{cases}\] 
We then define $W_{v}\subseteq V_{v}$ to be the orthogonal complement of $W^{v}$ with respect to $\langle -,- \rangle_{v}$. Let \[
W_{S'}:=\bigoplus_{v\in S'} W_{v}\ \ \ \ W^{S'}:=\bigoplus_{v\in S'} W^{v}.
\]
\end{definition}
\begin{remark}\label{remark: paring for W and I}
By construction we see that $W_{S'}$ is the orthogonal complement of $W^{S'}$ with respect to $\langle -,-\rangle_{S'}$ and vice versa. Furthermore, we have induced pairings
\[
    \langle -,-\rangle_{S'}^{1}:I_{S'}\times W^{S'}\rightarrow \mathbb{Z}/2\mathbb{Z}, \ \langle -,-\rangle_{S'}^{2}:W_{S'}\times I^{S'}\rightarrow \mathbb{Z}/2\mathbb{Z}.
\]
\end{remark}

\begin{proposition}[{\cite[Prop 2.2.4]{H19}}]\label{prop: Harpaz selmer conditions}
The Selmer group $\Sel(\mathcal{T},S)$ can be identified with \begin{enumerate}
    \item $I_{S'}\cap W_{S'}$,
    \item The left kernel of $\langle -,-\rangle_{S'}^{1}$,
    \item The left kernel of $\langle -,-\rangle_{S'}^{2}$.
\end{enumerate}
The dual Selmer group $\Sel(\hat{\mathcal{T}},S)$ can be identified with \begin{enumerate}
    \item $I^{S'}\cap W^{S'}$,
    \item The right kernel of $\langle -,-\rangle_{S'}^{1}$,
    \item The right kernel of $\langle -,-\rangle_{S'}^{2}$.
\end{enumerate}
\end{proposition}
The following is a key Proposition we will use to establish Theorem \ref{thm: main thm}.

\begin{proposition}[{\cite[Prop 2.3.1, Corollary 2.3.2]{H19}}]\label{prop: harpaz result for hasse principle and size of selmer group} If for all $v\not\in S_{0}$ we have $\valv(d)\leq 1$ and $\valv(d)=1$ if $v$ lies above $2$, then for all $a\mid d$ the $\mathcal{O}_{S_{0}}$-scheme $\mathcal{U}_{a,b}$ has an $\mathcal{O}_{S_{0}}$-point if and only if the $\mathcal{O}_{S}$ scheme $\mathcal{U}_{a,b}\times_{\Spec \mathcal{O}_{S_{0}}} \Spec \mathcal{O_{S}}$ has an $\mathcal{O}_{S}$-point. Moreover, if $\Sel(\hat{\mathcal{T}},S)$ is generated by $[d]$ then for every $a\mid d$ the $\mathcal{O}_{S_{0}}$ scheme $\mathcal{U}_{a,b}$ satisfies the $S_{0}$-integral Hasse principle.

\end{proposition}

\section{Condition (D)}\label{sec: condition (D)}
We formulate Condition (D), this property is crucial as it allows one to be able to kill certain elements in the Selmer and dual Selmer group. Through Section \ref{sec: condition (D)} $a,b,d_{i},c_{i}$ and $p_{i}$ will be as in Theorem \ref{thm: main thm}.

\begin{definition} Denote by $G$ the abelian group \[
G:=\mathbb{Q}^{*}/(\mathbb{Q}^{*})^2\bigoplus \spann_{\mathbb{F}_{2}}\{[p_{i}]:i\in J\} 
\] where $[p_{i}]$ is a formal symbol. We will write elements of $G$ as $[c][p_{J'}]$ where and $J'\subseteq J$ and \[
[c][p_{J'}]:=c\prod\limits_{i\in J'}p_{i}.
\] 
\end{definition}
\begin{notation}
Let $J'\subseteq J$ then we denote by $(J')^{c}$ the complement of $J'$ in $J$ i.e. $(J')^{c}:=J\backslash J'$.
\end{notation}
\begin{definition}\label{defn: lower G}
Let $i\in J$, we denote by $G_{i}\subseteq G$ the subgroup containing the elements $[c][p_{J'}]$ such that:
\begin{enumerate}
\item $\left[cp_{J'}(-\frac{d_{i}}{c_{i}})\right]\in \left\{1,[ap_{A}(-\frac{d_{i}}{c_{i}})]\right\}$ if $i\not \in J'$ and $i\not \in A$, \\
\item  $\left[cp_{J'}(-\frac{d_{i}}{c_{i}})\right]\in \left\{1,[bp_{B}(-\frac{d_{i}}{c_{i}})]\right\}$ if $i\not \in J'$ and $i\in A$,
\item $\left[cdp_{(J')^{c}}(-\frac{d_{i}}{c_{i}})\right]\in \left\{1,[ap_{A}(-\frac{d_{i}}{c_{i}})]\right\}$ if $i \in J'$ and $i\not \in A$, \\
\item  $\left[cdp_{(J')^{c}}(-\frac{d_{i}}{c_{i}})\right]\in \left\{1,[bp_{B}(-\frac{d_{i}}{c_{i}})]\right\}$ if $i \in J'$ and $i\in A$.
\end{enumerate}

\end{definition}
\begin{definition}\label{defn: upper G}
Let $i\in J$, we denote by $G^{i}\subseteq G$ the subgroup containing the elements $[c][p_{J'}]$ such that:
\begin{enumerate}
\item $\left[cp_{J'}(-\frac{d_{i}}{c_{i}})\right]\in \left\{1,[ap_{A}(-\frac{d_{i}}{c_{i}})]\right\}$ if $i\not \in J'$ and $i\not \in A$, \\
\item  $\left[cp_{J'}(-\frac{d_{i}}{c_{i}})\right]\in \left\{1,[bp_{B}(-\frac{d_{i}}{c_{i}})]\right\}$ if $i\not \in J'$ and $i\in A$,
\item $\left[-cdp_{(J')^{c}}(-\frac{d_{i}}{c_{i}})\right]\in \left\{1,[ap_{A}(-\frac{d_{i}}{c_{i}})]\right\}$ if $i\in J'$ and $i\not \in A$, \\
\item  $\left[-cdp_{(J')^{c}}(-\frac{d_{i}}{c_{i}})\right]\in \left\{1,[bp_{B}(-\frac{d_{i}}{c_{i}})]\right\}$ if $i\in J'$ and $i\in A$.
\end{enumerate}
\end{definition}
\begin{definition}
Denote by $G_{D}$ the intersection $G_{D}:=\bigcap_{i\in J}G_{i}$. Similarly, denote by $G^{D}$ the intersection $G^{D}:=\bigcap_{i\in J}G^{i}$.
\end{definition}
\begin{condition} The group $G_{D}$ is generated by $[a][p_{A}]$ and $[d][p_{J}]$ and $G^{D}$ is generated by $[-d][p_{J}]$.
\end{condition}
It will be easier for future calculations to use the notation of Harpaz \begin{notation} Let $d=ab$ and $J'\subseteq J$, then for $i\in J$ denote by \[
D_{i}^{J'}:=\begin{cases}
p_{J'}\left(-\frac{d_{i}}{c_{i}}\right) &\text{ if } i \not\in J',\\
dp_{(J')^{c}}\left(-\frac{d_{i}}{c_{i}}\right) &\text{ if } i \in J'.
\end{cases} \ \ \hat{D}_{i}^{J'}:=\begin{cases}
p_{J'}\left(-\frac{d_{i}}{c_{i}}\right) &\text{ if } i \not\in J',\\
-dp_{(J')^{c}}\left(-\frac{d_{i}}{c_{i}}\right) &\text{ if } i \in J'.
\end{cases}
\] Then Definitions \ref{defn: lower G},\ref{defn: upper G} can be reformulated in the above notation: $G_{i}$ is the subgroup of $G$ containing the elements $[c][p_{J'}]$ such that \[
[cD^{J'}_{i}]\in \langle [aD^{A}_{i}]\rangle
\] and $G^{i}$ is the subgroup of $G$ containing the elements $[c][p_{J'}]$ such that \[
[c\hat{D}^{J'}_{i}]\in \langle [aD^{A}_{i}]\rangle.
\]

\end{notation}
\section{Schinzel Hypothesis}\label{sec: Schinzel's Hypothesis}
\begin{conjecture}[Schinzel's Hypothesis (H)]\label{conjecture: Schinzel's Hypothesis}
For every finite collection $\{f_{1},\dotsc,f_{k}\}$ of non-constant irreducible polynomials over the integers with positive leading coefficients, one of the following conditions holds:\begin{enumerate}
    \item There is an integer $m$ (called a fixed divisor) which always divides the product $f_{1}(n)\dotsc f_{k}(n)$ for $n\in \mathbb{Z}$. Or, equivalently: There exists a prime $p$ such that for every $n$ there exists an $i$ such that $f_{i}(n)\equiv 0$ (mod $p$).
    \item There are infinitely many positive integers $n$ such that $f_{1}(n),\dotsc,f_{k}(n)$ are simultaneously prime numbers.
\end{enumerate}
\end{conjecture}
\begin{remark} 
Consider a finite collection of linear polynomials $\{f_{1}(t),\dotsc,f_{k}(t)\}$ where $f_{i}(t)=c_{i}t+d_{i}$ and \[\Delta_{j,k}:=c_{j}d_{k}-c_{k}d_{j}\neq 0\] for $j\neq k$. Then under Hypothesis $(\text{H})$ there exists infinitely many integers $n$ such that $f_{1}(n),\dotsc ,f_{k}(n)$ are distinct primes.
\end{remark}

We will require a slightly different version of (H) due to Serre.
\begin{conjecture}[Serre ($\text{H}_1$)]\label{conjecture: Schinzel's Hypothesis H1} Let $k$ be a number field, $f_{i}(t)$ irreducible polynomials over $k$. Let $S$ be a finite set of places of $k$, containing all \begin{enumerate}
    \item archiemedean places,
    \item finite places $v$ where one of $f_{i}$ does not have $v$-integral coefficients or has all its coefficients divisible by $v$,
    \item all finite places above a prime $p$ less than or equal to the degree of the polynomial $N_{k/\mathbb{Q}}(f_{i}(t))$.
\end{enumerate} Given elements $\lambda_{v}\in k_{v}$ at finite places $v\in S$ one may find $\lambda \in k$, integral away from $S$, arbitrarily close to each $\lambda_{v}$ for the $v$-adic topology for the $v$-adic topology for $v\in S$ finite, arbitrarily big in the archidmedian completions $k_{v}$, and such that for each $i$, $f_{i}(\lambda)\in k$ is a unit in $k_w$ for all place $w\not\in S$ except perhaps one place $w_{i}$, where it is a uniformizer.
\end{conjecture}
\begin{remark}
Conjecture \ref{conjecture: Schinzel's Hypothesis} implies Conjecture \ref{conjecture: Schinzel's Hypothesis H1}, \cite[Lemma 4.1]{CTSD94}. Because of this we will denote by (H) the conjectures (H) and $(\text{H}_{1})$ as they are equivalent.
\end{remark}
\section{Vertical Brauer group}\label{sec: vertical Brauer group}
In this section we will determine the vertical Brauer group of the scheme \[
U:ap_{A}(t)x^2+bp_{B}(t)y^2=1\subseteq \mathbb{A}^{3}_{\mathbb{Q}}
\] with respect to the morphism $\pi:U\rightarrow \mathbb{A}^{1}, (x,y,t)\mapsto t$ and keeping notation as in Theorem \ref{thm: main thm}.

\begin{definition}
Let $f:X\rightarrow Y$ be a morphism of varieties over a field $k$. The \emph{vertical Brauer group} $\Br^{\text{vert}}(X,f)$ is defined as:
\[
\Br^{\text{vert}}(X,f):=\Br(X)\cap f^{*}\Br(k(Y))\subseteq \Br(k(X)).
\]
\end{definition}
\begin{definition}
A scheme $S$ of finite type over a perfect field $k$ is called \emph{split} if $S$ contains an irreducible component of multiplicity 1 which is geometrically irreducible.
\end{definition}
To find the the vertical Brauer group of $U$ we will invoke a consequence of Grothendieck's purity theorem.
\begin{proposition}[{\cite[Thm 3.7.3]{CS21}}]\label{prop: residue maps}
Let $V$ be a smooth integral variety over a field $k$ of characteristic $0$. Then there exists an exact sequence \[
0\rightarrow \Br V\rightarrow \Br k(V) \xrightarrow[]{\partial_{D}} \bigoplus_{D}\HH^{1}_{\et}(k(D),\mathbb{Q}/\mathbb{Z})
\] where $D$ ranges over irreducible divisors of $V$. We call the maps $\partial_{D}$ \emph{residue maps}.
\end{proposition}
\begin{remark}
Let $V$ be as in Proposition \ref{prop: residue maps}. We denote by $(f,g)$ the Azumaya algebra \[
(f,g):=\bigoplus_{0\leq a,b<2}u^{a}v^{b}k(X) \text{ where } u^{2}=f, v^{b}=g, vu=-uv .
\] Let $D$ be an integral divisor on $V$, then by Kummer theory \[
\HH^{1}_{\et}(k(D),\mathbb{Z}/2\mathbb{Z})\cong k(D)^{*}/(k(D)^{*})^2.
\] Given a class $(f,g)\in \Br k(V)$ we can give an explicit description of $\partial_{D}((f,g))$ in $k(D)^{*}/(k(D)^{*})^2$ via the tame symbol \cite[Ex 7.1.5, Prop 7.5.1]{GT06}. Let $x$ be the generic point of the divisor $D$, then as $D$ is an irreducible divisor on $V$ we have $k(D)=\mathcal{O}_{D,x}=k(x)$. We denote by bar the image of the map $\mathcal{O}_{X,x}\rightarrow k(x)$. Let $(f,g)\in \Br k(V)$ then 
\[
\partial_{D}\left( (f,g) \right) = (-1)^{\valx (f)\valx (g)}\left(\overline{\frac{f^{\valx (g)}}{g^{\valx (f)}}}\right) \in k(D)^{*}/(k(D)^{*})^{2}
\] If $\partial_{D}\left( (f,g) \right)=0$ then we say $(f,g)$ is \emph{unramified} at $D$.
\end{remark}
\begin{definition}
For each $i\in J$, let $\mathcal{A}_{i}\in \Br k(U)$ be 
\[
\mathcal{A}_{i}:=\begin{cases}
( ap_{A}(-\frac{d_{i}}{c_{i}}),p_{i}(t)) \ \ &\text{if}\ i\not\in A,\\
( bp_{B}(-\frac{d_{i}}{c_{i}}),p_{i}(t)) \ \ &\text{if}\ i\in A.\\
\end{cases}
\]
\end{definition}
\begin{lemma}\label{lemma: gen of vertical Brauer group} The class $\mathcal{A}_{i}$ lies in $\Br^{\text{Vert}}(U,\pi)$.
\end{lemma}
\begin{proof}
As the left entry of $\mathcal{A}_{i}$ is constant, for any prime divisor $D\subseteq U$ with generic point $\eta_{D}$, we have 
\begin{equation}\label{eq: residue away from divisor inside zero locus}
\text{val}_{\eta_{D}}\left(ap_{A}\left(-\frac{d_{i}}{c_{i}}\right)\right)=\text{val}_{\eta_{D}}\left(bp_{B}\left(-\frac{d_{i}}{c_{i}}\right)\right)=0.
\end{equation} Let $Z_{i}\subseteq U$ be the vanishing locus of $p_{i}$, by (\ref{eq: residue away from divisor inside zero locus}) $\mathcal{A}_{i}$ is unramified outside of $Z_{i}$. Suppose $i\not\in A$, for $D\subseteq Z_{i}$ we have \[
\partial(\mathcal{A}_{i})_{\eta_{D}}=\frac{1}{ap_{A}(\frac{-d_{i}}{c_{i}})^{\text{val}_{\eta_{D}}(p_{i}(t))}}.\] As $bp_{B}(t)\mid_{{Z}_{i}}=0
$ then $ap_{A}(t)|_{Z_{i}}$ is a square in $k[Z_{i}]$. Moreover, as $ap_{A}(t)|_{Z_{i}}=ap_{A}(-\frac{d_{i}}{c_{i}})$, we have $\mathcal{A}_{i}$ is unramified along $Z_{i}$. The case for $i\in A$ is similar. It is then clear that $\mathcal{A}_{i}\in \Br^{\text{Vert}}(U,\pi)$.
\end{proof}
 \begin{remark} Let $m_{i}:=\{t_{i}\in \mathbb{A}^{1}:p_{i}(t_{i})=0\}$. Assume $i\not\in A$, then the fibre of $\pi$ above $m_{i}$ is \[
ap_{A}(t_{i})x^2=1.
 \] Suppose $ap_{A}(t_{i})$ is not a square over $k(m)$, then one can define a morphism \[
 \mathbb{A}^{1}_{L}\rightarrow  U_{m_{i}},\ y\mapsto (1/\alpha,y)
 \] where $k(\alpha)=k[w]/(w^2-ap_{A}(t_{i}))$, in this case the fibre $U_{m_{i}}$ is spit after a quadratic extension. If $ap_{A}(t_{i})$) is a square one can take $k(\alpha)=k$ and we have two copies of $\mathbb{A}^{1}$, namely $(1/\alpha,y)$ and $(-1/\alpha,y)$, each are geometrically irreducible i.e. the fibre in this case is split. 
 \end{remark}
 For the rest of Section \ref{sec: vertical Brauer group}, let $m\in \mathbb{A}^{1}$ be a closed point, $U_{m}$ the fibre of $\pi:U\rightarrow \mathbb{A}^{1}$ above $m$ and $\mathcal{A}\in \Br k(t)$ such that $\pi^{*}\mathcal{A}\in \Br^{\text{vert}}(U,\pi)$.
\begin{lemma}\label{lemma: vertical Brauer group, smooth fibre}
Suppose the fibre $U_{m}$ is smooth, then $\partial_{m}(\mathcal{A})=0$.
 \end{lemma}
 \begin{proof}
As the fibre is smooth, hence split we have a commutative diagram \[
\begin{tikzcd}
0 \arrow[r] & \Br k \arrow[r] \arrow[d] & \Br k(t) \arrow[r,"\partial_{m}"] \arrow[d] & \HH^{1}(k(m),\mathbb{Z}/2\mathbb{Z}) \arrow[r,"\sim"] \arrow[d] & k(m)^{*}/(k(m)^{*})^2 \arrow[d]\\
0 \arrow[r] & \Br U \arrow[r]           & \Br k(U) \arrow[r,"\partial_{U_{m}}"]           & \HH^{1}(k(U_{m}),\mathbb{Z}/2\mathbb{Z}) \arrow[r,"\sim"]           & k(U_{m})^{*}/(k(U_{m})^{*})^2.
\end{tikzcd}
\] Then the image of $\partial_{U_{m}}(\pi^{*}\mathcal{A})$ is equal to the image of $\partial_{m}\mathcal{A}$ under the map $k(m)^{*}/(k(m)^{*})^2\rightarrow k(U_{m})^{*}/(k(U_{m})^{*})^2$. By assumption $\partial_{U_{m}}(\pi^{*}\mathcal{A})=0$, hence $\partial_{m}\mathcal{A}$ is a square in $k(U_{m})^{*}/(k(U_{m})^{*})^2$. As $U_{m}$ is smooth affine conic it is geometrically integral, so the algebraic closure of $k(m)$ in $k(U_{m})$ is $k(m)$, hence $\partial_{m}\mathcal{A}$ is a square.
 \end{proof}
  \begin{lemma}
Suppose the fibre $U_{m}$ is a split singular fibre, then $\partial_{m}(\mathcal{A})=0$.
 \end{lemma}
 \begin{proof}
 If $U_{m}$ is split over $k(m)$ then take a component of multiplicity $1$, say $Y_{m}$. Applying the same argument as Lemma \ref{lemma: vertical Brauer group, smooth fibre} but with $Y_{m}$ instead of $U_{m}$ we get the statement.
 \end{proof}
   \begin{lemma}
Suppose $U_{m}$ is a singular fibre that is not split, then $\partial_{m}(\mathcal{A})=\partial_{m}(\mathcal{A}_{i})$ for some $i\in J$ or $\partial_{m}(\mathcal{A})=0$.
 \end{lemma}
 \begin{proof}
The singular fibres of $\pi:U\rightarrow \mathbb{A}^{1}$ are exatcly \[
\mathcal{M}:=\{m_{i}\in \mathbb{A}^{1}(k):p_{i}(m_{i})=0\}.
\] Then $\partial_{m_{i}}\mathcal{A}_{i}=\alpha_{i}\in k^{*}/(k^{*})^2$, where $\alpha_{i}$ is such that the fibre $U_{m}$ splits over $k(\sqrt{\alpha_{i}})$. Moreover, it is clear that $k(U_{m})=k(t)(\sqrt{\alpha_{i}})$ where $t$ is a purely transcendental element. Then $\partial_{U_{m}}\pi^{*}\mathcal{A}=0$ if and only if $\partial_{m}\mathcal{A}=0$ or $\partial_{m}\mathcal{A}=\alpha_{i}$ in $k^{*}/(k^{*})^2$
 \end{proof}
 \begin{proposition}\label{prop: vertical Brauer group}
We have that $\Br^{\text{Vert}}(U,\pi)/\Br k$ is generated by the $\mathcal{A}_{i}$ i.e. if $\mathcal{A}\in \Br^{\text{Vert}}(U,\pi)$ then $\mathcal{A}=\sum\limits_{i\in J} \epsilon_{i}\mathcal{A}_{i} + \mathcal{B}$ where $\epsilon_{i}\in \mathbb{Z}/2\mathbb{Z}$ and $\mathcal{B}\in \Br k$.
 \end{proposition}
 \begin{proof}
Let $\mathcal{A}\in \Br^{\text{Vert}}(U,\pi)$ which is the pull back of an element $\mathcal{A}'\in \Br k(t)$. Then $\partial_{m}(\mathcal{A}')=0$ for all but finitely many closed points $m\in \mathbb{A}^{1}$, hence $\mathcal{A}'\in \Br(W)$ where $W$ is a dense open of $\mathbb{A}^{1}$. It then follows that $\mathcal{A}\in \Br(\pi^{-1}W)$ i.e. non-trivial residues of $\mathcal{A}$ can only appear from vertical divisors $D\subset U$, by definition these correspond to the fibres of $\pi:U\rightarrow \mathbb{A}^{1}$. Let 
\[
E:=\{i:\mathcal{A}' \text{ has non-trivial residue at the closed point } m_{i}\in \mathbb{A}^{1}\}.
\] For each $i\in E$ we have that $\mathcal{A}'$ and $\mathcal{A}_{i}$ have the same residue at $m_{i}$, hence $\alpha:=\mathcal{A}'-\sum_{i\in E}\mathcal{A}_{i}\in \Br \mathbb{A}^{1}=\Br k$ and $\pi^{*}(\mathcal{A}'-\sum_{i\in E}\mathcal{A}_{i})=\mathcal{A}-\sum_{i\in E}\mathcal{A}_{i}=0$ in $\Br^{\text{Vert}}U/\Br k$.
 \end{proof}
 \begin{notation}
Let $v$ be a place of $\mathbb{Q}$. Throughout the rest of this article, given a quaternion algebra $(a,b)\in \Br \mathbb{Q}_{v}$ we denote by $\langle a,b\rangle_{v}$ the image of the invariant map $v$ i.e. $\inv_{v}(a,b)=\langle a,b\rangle_{v}$.
 \end{notation}
 \section{Locally soluble fibres}
 Through out the rest of the paper $S_{0}$ will be the finite set of places in Theorem \ref{thm: main thm}.
 \begin{definition}
Let $S_{\text{bad}}\subseteq \Omega_{\mathbb{Q}}\backslash S_{0}$ be the set of places outside of $S_{0}$ for which \begin{enumerate}
    \item $v\mid \Delta_{i,j}$ for some $i\neq j$,
    \item $v\mid d$,
    \item $v=2$,
    \item $v$ such that $\valv(p_{J}(t_{v}))>0$ for every $t_{v}\in \mathbb{Z}_{v}$.
\end{enumerate}
\end{definition}
\begin{notation}
For $a\in \mathbb{Z}^{*}_{v}$ for some place $v$, we shall denote by $[a]_{v}$ for the class of $a$ in $\mathbb{Z}_{v}^{*}/(\mathbb{Z}_{v}^{*})^2$.
\end{notation}
\begin{lemma}\label{lemma: locally soluble fibre} Let $v\not\in S_{0}\cup S_{bad}$ be a place. Suppose $t_{v}\in \mathbb{Z}_{v}^{*}$ and $i\in J$ is an element such that $\valv(p_{i}(t_{v}))>0$. Denote by $\mathcal{U}_{t_{v}}$ of the fibre above $t_{v}$, then $\mathcal{U}_{t_{v}}(\mathbb{Z}_{v})\neq \emptyset$ if and only if \begin{enumerate}
\item $[ap_{A}(-\frac{d_{i}}{c_{i}})]_{v}=0 \ \text{if}\ i\not\in A$,
\item $[bp_{B}(-\frac{d_{i}}{c_{i}})]_{v}=0 \ \text{if}\ i\in A.$
\end{enumerate}
\end{lemma}
\begin{proof}
Note as $v\not\in S_{\text{bad}}$ we have that $\valv(p_{i}(t_{v}))>0$ for a unique $i\in J$. We shall assume that $i\not\in A$, then the special fibre of $\mathcal{U}_{t_{v}}\rightarrow \Spec \mathbb{Z}_{v}$ is defined by the equation
\[
\mathcal{U}_{t_{v}}: ap_{A}\left(-\frac{d_{i}}{c_{i}}\right)y^2=1\ (\text{mod}\ \mathfrak{m}_{v}).
\] By Hensel's Lemma $\mathcal{U}_{t_{v}}(\mathbb{Z}_{v})\neq \emptyset$ if and only if $\left[ap_{A}\left(-\frac{d_{i}}{c_{i}}\right)\right]_{v}=0$. The case of $i\in A$ is similar.
\end{proof}
\begin{lemma}\label{lemma: Brauer element pairing trivially at place} Let $v\not\in S_{0}\cup S_{bad}$ be a place. Then for $i\in J$ we have that \[
\inv_{v}\mathcal{A}_{i}(P_{v})=0
\] for all $P_{v}=(x_{v},y_{v},t_{v})\in \mathcal{U}(\mathbb{Z}_{v})$.
\end{lemma}
\begin{proof}
As $\inv_{v}\mathcal{A}_{i}(P_{v}) = \langle aD^{A}_{i},p_{i}(t_{v}) \rangle_{v}$ and $aD_{i}^{A}$ is a $v$-adic unit, we see $\inv_{v}\mathcal{A}_{i}(P_{v})\neq 0$ if and only if $aD_{i}^{A}$ is a not a square and $\valv(p_{i}(t_{v}))\equiv 1$ mod $2$. By Lemma \ref{lemma: locally soluble fibre} this is impossible as $\mathcal{U}_{t_{v}}(\mathbb{Z}_{v})\neq \emptyset$.
\end{proof}
\section{Partial adelic points}
\begin{definition} Let $T\subset \Omega_{\mathbb{Q}}$ be a finite set of places containing $S_{0}$ then a \emph{partial adelic point over $T$} is a point $(P_{v})_{v\in T}=(x_{v},y_{v},t_{v})_{v\in T}$ in the product \[
(P_{v})_{v\in T}\in \prod\limits_{v\in S_{0}} \mathcal{U}(\mathcal{O}_{v})\times \prod\limits_{v\in T \backslash S_{0}}U(k_{v}).
\] Moreover, we denote by $P_{T}$ the partial adelic point $(P_{v})_{v\in T}$ over $T$.
\end{definition}
\begin{remark}\label{rem: construction}
If $[x][p_{J'}]\in G_{i}\backslash G_{D}$ i.e. there exists $i_{x}\neq i$ such that $cD^{J'}_{i_{x}}\not\in \{1,aD^{A}_{i_{x}}\}$, then there exists a place $v_{x}\not\in S_{0}\cup S_{\text{bad}}$ such that $cD^{J'}_{i_{x}}$ is not a square at $v_{x}$ but $aD^{A}_{i_{x}}$ is a square at $v_{x}$. For each $i$ and each $x\in G_{i}\backslash G_{D}$ choose a prime $v_{x}$ as above. Then one obtains a finite set of $v_{x}$'s which we will denote by $S_{D}$.
\end{remark}
\begin{definition}\label{definition: defn of S} Let $S:=S_{0}\cup S_{\text{bad}}\cup S_{D}$.
\end{definition}

\begin{definition}\label{defn: suitable partial adelic point} Let $T$ be a finite set of places containing $S$. A partial adelic point $P_{T}$ is called \emph{suitable} if the following hold:
\begin{enumerate}
\item $dp_{J}(t_{v})\neq 0$ for all $v\in T$,
\item $\valv(dp_{J}(t_{v}))\leq 1$ for all $v\in T \backslash S_{0}$,
\item $\valtwo(dp_{J}(t_{v}))=1$ if $2\in T \backslash S_{0}$,
\item There is one place $v\in S_{0}$ such that $-dp_{v}(t_{v})\in \mathbb{Q}_{v}^{*}$ is a square,
\item $\sum\limits_{v}\inv_{v}\mathcal{A}_{i}(P_{v})=0$ for every $i\in J$.
\item For every $v_{x}\in S_{D}$ the local point $P_{v}\in \mathcal{U}(\mathbb{Z}_{v})$ lies above $t_{x}\in \mathbb{A}^{1}(\mathbb{Z}_{v_{x}})$ as described in remark \ref{rem: construction}.
\end{enumerate}
\end{definition}
\begin{proposition}\label{prop: Can always find suitable partial adelic point}
Under the assumptions of Theorem \ref{thm: main thm} there always exists a suitable partial adelic point with respect to $T$ on $\mathcal{U}$.
\end{proposition}
\begin{proof}
Under the assumptions of Theorem \ref{thm: main thm} we have an adelic point satisfying (1)-(4) of Definition \ref{defn: suitable partial adelic point}. As this adelic point is orthogonal to the vertical Brauer group by Lemma \ref{lemma: gen of vertical Brauer group} we have condition (5) is satisfied. For such $v_{x}$ the invariant map at $v_{x}$ pairs trivially with $\mathcal{U}(\mathbb{Z}_{v_{x}})$ by Lemma \ref{lemma: Brauer element pairing trivially at place} we can shift our adelic point without disrupting the properties (1)-(5) and satisfying $(6)$.
\end{proof}
\section{Admissible fibres}
\begin{definition}\label{defn: admissible points}
Given a finite set of places $T$ containing $S$, and a suitable partial adelic point $P_{T}\in \mathcal{U}$, we will say that a point $t_{0}\in \mathbb{Z}_{S_{0}}$ is \emph{$T$-admissible} with respect to $P_{T}$ if \begin{enumerate}
\item $[p_{i}(t_{0})]_{v}=[p_{i}(t_{v})]_{v}$ for all $v\in T$.
\item For each $i\in J$, we have that $p_{i}(t)$ is a unit outside of $T$ except at one place $u_{i}$, where $p_{i}(t_{0})$ is a uniformiser (i.e. $\valu(p_{i}(t_{0}))=1$).
\item The fibre $\mathcal{U}_{t_{0}}$ has an $S_{0}$-integral adelic point i.e. $\mathcal{U}_{t_{0}}(\mathbb{A}_{S_{0}})\neq \emptyset$.
\end{enumerate}
\end{definition}
\begin{definition}
Let $P_{T}=(x_{v},y_{v},t_{v})_{v\in T}$ be a suitable partial adelic point then we define $T_{0}\subseteq T$ to be the subset of places $v$ which lie in $S_{0}$ or $\valv(dp_{J}(t_{v}))=1$.
\end{definition}
\begin{definition}
Let $t$ be a $T$-admissible point with  $\{u_{i}\}_{i\in J}$ the places appearing in condition $(2)$ of Definition \ref{defn: admissible points}, then we denote by $T(t):=T\cup \{u_{i}\}_{i\in J}$ and $T_{0}(t):=T_{0}\cup \{u_{i}\}_{i\in J}$.
\end{definition}
\begin{proposition}\label{prop: can always find T-admissible point} Let $T$ be a finite set of places containing $S$ and $P_{T}$ a suitable partial adelic point with respect to $T$ on $\mathcal{U}$. Then there exists a $T$-admissible point $t$ on $\mathbb{A}^{1}_{\mathbb{Z}_{S_{0}}}$ with respect to $P_{T}$.
\end{proposition}
\begin{proof}
By Conjecture \ref{conjecture: Schinzel's Hypothesis H1} we can find $t\in \mathbb{Z}_{S_{0}}$ and places $u_{i}$ for $i\in J$ such that conditions $(1)-(2)$ of Definition \ref{defn: admissible points} are satisfied. Now it is sufficient to show that the fibre above $t$ has an $S_{0}$-integral adelic point. One applies the inverse function theorem \cite[pg. 73]{Ser64} to see that for all $v\in T$ we have $\mathcal{U}_{t}(\mathbb{Z}_{v})\neq \emptyset$. For $v\not\in T\cup \{u_{i}\}_{i\in J}$ the special fibre of $\mathcal{U}_{t}\rightarrow \Spec \mathbb{Z}_{v}$ is smooth and has a point over $\mathbb{F}_{v}$ (as it is a smooth affine conic), hence $\mathcal{U}_{t}(\mathbb{Z}_{v})\neq \emptyset$ by Hensel's Lemma. We now focus on $v=u_{i}$ for some $i\in J$. By quadratic reciprocity \[
\langle aD^{A}_{i},p_{i}(t)\rangle_{u_{i}}=\sum_{v\in T}\langle aD^{A}_{i},p_{i}(t)\rangle_{v}=\sum_{v\in T}\langle aD^{A}_{i},p_{i}(t_{v})\rangle_{v}.
\] By property $(5)$ of Definition \ref{defn: suitable partial adelic point} we have that this sum is zero. As $p_{i}(t)$ is a uniformizer at $u_{i}$ and $aD^{A}_{i}$ is a $u_{i}$-adic unit, it follows that $aD^{A}_{i}$ is a square, hence by Lemma \ref{lemma: locally soluble fibre}, $\mathcal{U}_{t}(\mathbb{Z}_{u_{i}})\neq \emptyset$.
\end{proof}
 \section{Selmer group of admissible fibres}
Give a suitable partial adelic point $(P_{v})_{v\in T}$ and a $T$-admissible point $t_{0}$ of $P_{T}$ we want to be able to control the size of $\Sel(\hat{\mathcal{T}}_{t_{0}},T_{0}(t_{0}))$. However, the first natural question is: 
\begin{question}
Given two $T$-admissible points $t_{0},t_{1}$ of $P_{T}$, how do $\dim_{\mathbb{F}_{2}}\Sel(\hat{\mathcal{T}}_{t_{0}},T_{0}(t_{0}))$ and $\dim_{\mathbb{F}_{2}}\Sel(\hat{\mathcal{T}}_{t_{1}},T_{0}(t_{1}))$ differ? 
\end{question}
In this section we show that the choice of $T$-admissible point $t_{0}$ has no effect on dimension of $\Sel(\hat{\mathcal{T}}_{t_{0}},T_{0}(t_{0}))$ as a $\mathbb{F}_{2}$-vector space. 

\begin{definition}
Let $T$ be a finite set of places of $\mathbb{Q}$ containing $S$. Let $J_{T}\subseteq G$ generated by $I_{T}$ and the symbols $[p_{i}]$. Similarly, $J^{T}\subseteq G$ generated by $I^{T}$ and the symbols $[p_{i}]$ i.e. \[
J^{T},J_{T}:=\mathcal{O}^{*}_{T}/(\mathcal{O}^{*}_{T})^{2}\bigoplus \text{Span}_{\mathbb{F}_{2}}\{[p_{i}]:i\in J\}
\]
\end{definition}
\begin{definition}
For a $T$-admissible point $t_{0}$, we denote by $ev_{t_{0}},ev^{t_{0}}$ the maps \[
ev_{t_{0}}:J_{T}\rightarrow I_{T(t_{0})}, \ \ [c][p_{J'}]\mapsto [c][p_{J'}(t_{0})]
\] and \[
ev^{t_{0}}:J^{T}\rightarrow I^{T(t_{0})}, \ \ [c][p_{J'}]\mapsto [c][p_{J'}(t_{0})]
\]
\end{definition}
\begin{remark}
$ev_{t_{0}}$ and $ev^{t_{0}}$ are isomorphisms of $\mathbb{F}_{2}$-vector spaces.
\end{remark}
\begin{definition}
We will define the \emph{relative Selmer group} and the \emph{relative dual Selmer group} as\[
R_{P_{T},t_{0}}:=ev_{t_{0}}^{-1}\left(\Sel(\mathcal{T}_{t_{0}},T_{0}(t_{0}))\right) \ \text{and} \ \hat{R}_{P_{T},t_{0}}:=ev_{t_{0}}^{-1}\left(\Sel(\hat{\mathcal{T}}_{t_{0}},T_{0}(t_{0}))\right).
\]
\end{definition}
\begin{remark}
Recall Definition \ref{defn: definition of Wv}. In the rest of the paper we use the following analogous notation: given a $T$-admissible point $t$ and a place $v\in T(t)$ we denote by $W^{v}(t)\subseteq V^{v}$ for the space 
\[ W^{v}(t):=\begin{cases} [-dp_{J}(t)] \ &\text{if}\ v\in T_{0}(t),\\
\im\left(\mathcal{O}_{v}^{*}/(\mathcal{O}_{v}^{*})^{2}\rightarrow k_{v}^{*}/(k_{v}^{*})^2\right)\ &\text{otherwise}.
\end{cases}\] 
We then define $W_{v}(t)\subseteq V_{v}$ to be the orthogonal complement of $W^{v}(t)$ with respect to $\langle -,- \rangle_{v}$.
\end{remark}
\begin{lemma}\label{lemma: selmer conditions}
  Given a finite set $T\subset \Omega_{\mathbb{Q}}$ containing $S$ and $t_{0}$ a $T$-admissible point with respect to a suitable partial adelic point $P_{T}$. Let $i\in J$ and $v\in T_{0}(t_{0})$ where $v\not\in S_{0}\cup S_{\text{bad}}$ and odd, such that $t_{0}\equiv -\frac{d_{i}}{c_{i}}$ mod $\mathfrak{m}_{v}$. If $x=[c][p_{J'}]\in J^{T}$ and $\valv(c)=0$ then $ev^{t_{0}}(x)\in W^{v}(t_{0})$ if and only if \begin{enumerate}
      \item $[ev^{t_{0}}(x)]_{v}=0$ if $i\not\in J'$,
      \item $[ev^{t_{0}}(x[-d][p_{J}])]_{v}=0$ if $i\in J'$.
  \end{enumerate} Furthermore,
  $ev_{t_{0}}(x)\in W_{v}(t_{0})$ if and only if \begin{enumerate}[resume]
\item $[ev_{t_{0}}(x)]_{v}$ if $i\not\in J'$,
\item $[ev^{t_{0}}(x[d][p_{J}])]_{v}=0$ if $i\in J'$.
\end{enumerate}
\end{lemma}
\begin{proof}
For $i\neq j$ \[
p_{j}(t_{0})\equiv c_{j}t_{0}+d_{j}\equiv  -c_{j}\frac{d_{i}}{c_{i}}+d_{j} \ \text{mod}\ \mathfrak{m}_{v}
\] If $p_{j}(t_{0})\equiv 0\ \text{mod}\ \mathfrak{m}_{v}$ we have that $\Delta_{i,j}\equiv 0$ mod $\mathfrak{m}_{v}$ but as $v\not\in S_{bad}$, we can deduce $\valv(p_{j}(t_{0}))= 0$ for $i\neq j$. By assumption on $v$ we have that $p_{i}(t_{0})\equiv 0$ mod $\mathfrak{m}_{v}$. As $P_{T}$ is a suitable partial adelic point $\valv(dp_{J}(t_{v}))\leq 1$ for $v\in T\backslash S_{0}$ and $t_{0}$ approximates $t_{v}$ for $v\in T$. We can conclude for $v\in T\backslash S_{0}$ that $\valv(dp_{J}(t_{0}))\leq 1$. If $v\not \in T$ i.e. it is one of the primes $u_{i}^{0}$ (we can assume that each $u_{i}^{0}$ is distinct for each $i$) using the definition of $T$-admissible point $\valuiz(dp_{J})\leq 1$. Moreover, $\valv(p_{i}(t_{0}))= \valv(dp_{J}(t_{0}))=1$ and \[
\valv(ev^{t_{0}}(x)):=\begin{cases} 1 & \text{if}\ i\in J',\\
0 & \text{if}\ i\not\in J'.
\end{cases}
\] Recall $W^{v}(t_{0})$ is the subgroup of $V^{v}$ generated by $[-dp_{J}]_{v}$ and as $v$ is odd $W_{v}(t_{0})$ it is the subgroup of $V_{v}$ generated by $[dp_{J}]_{v}$. Then $ev_{t_{0}}(x)\in W_{v}(t_{0})$ if and only if pairs (with the pairing from Definition \ref{defn: Hilbert symbol pairing}) trivially with $dp_{J}(t_{0})$. Similarly, for $ev^{t_{0}}(x)$. Assume for now $i\not\in J'$. We showed that $\valv(dp_{J}(t_{0}))=1$, hence it is a uniformiser of $\mathbb{Z}_{v}$, so computing the Hilbert symbol \[
\langle ev^{t_{0}}(x),dp_{J}(t_{0})\rangle_{v}:=\begin{cases} 0 & \text{if}\ [ev^{t_{0}}(x)]_{v}=0.\\
1 & \text{otherwise}.
\end{cases}
\] This completes (1). If $i\in J'$ then both $ev^{t_{0}}(x)$ and $dp_{J}(t_{0})$ are uniformisers and by \cite[\S 3, Prop 2i,2iv)]{S73}
\[
\langle cp_{J'}(t_{0}),dp_{J}(t_{0})\rangle_{v} = \langle dp_{J}(t_{0}),-cdp_{(J')^{c}}(t_{0})\rangle_{v}=\langle dp_{J}(t_{0}),\ev^{t_{0}}(x[-d][p_{J}])\rangle_{u_{i}}
\] Then 
\[
\langle ev^{t_{0}}(x),dp_{J}(t_{0})\rangle_{v}:=\begin{cases} 0 & \text{if}\ [ev^{t_{0}}(x[-d][p_{J}])]_{v}=0.\\
1 & \text{otherwise}.
\end{cases}
\] This completes (2). The proofs for (3) and (4) are completely analogous to (1) and (2).
\end{proof}
\begin{remark}
Taking into account Proposition \ref{prop: Harpaz selmer conditions}, we call the conditions from Lemma \ref{lemma: selmer conditions} Selmer conditions.
\end{remark}
\begin{lemma}\label{lemma: selmer group of admissible fibres 2}
Let $x=[c][p_{J'}]\in J^{T}$. Then $x\in \hat{R}_{P_{T},t_{0}}$ if and only if\begin{enumerate}
    \item For every $v\in T$: $ev^{t_{0}}(x)\in W^{v}(t_{0})$.
    \item For every $i\in J$: \begin{align*}
        \langle p_{i}(t_{0}),ev^{t_{0}}(x)\rangle_{u_{i}}&=0\ \text{if}\ i\not\in J',\\
        \langle p_{i}(t_{0}),ev^{t_{0}}(x[-d][p_{J}])\rangle_{u_{i}}&=0\ \text{if}\ i\in J'.
    \end{align*}
\end{enumerate}
\end{lemma}
\begin{proof}
Let $\alpha:=ev^{t_{0}}(x)$. Recall $\valu(p_{i}(t_{0}))=\valu(dp_{J}(t_{0}))=1$, hence
\[
dp_{J}(t_{0})=d\left(\prod\limits_{k\neq i}p_{k}(t_{0})\right)p_{i}(t_{0})
\] where $d\prod\limits_{k\neq i}p_{k}(t_{0})\in \mathbb{Z}_{u_{i}}^{*}$. Moreover, $\valu(\alpha)=1$ if $i\in J$ and $0$ otherwise. If $i\not\in J$ then we know $\alpha$ is a unit and by \cite[\S3 Prop 2iii,Lemma 1.2]{S73}   \[
\langle\alpha,dp_{J}(t_{0})\rangle_{u_{i}}=\langle p_{i}(t_{0}),\alpha\rangle_{u_{i}}.
\] If $i\in J$ then by \cite[\S3 Prop 2iv]{S73} 
\[
\langle\alpha,dp_{J}(t_{0})\rangle_{u_{i}}=\langle p_{i}(t_{0}),-cdp_{(J')^{c}}(t_{0})\rangle_{u_{i}}=\langle p_{i}(t_{0}),ev^{t_{0}}(x[-d][p_{J}])\rangle_{u_{i}}.
\] 
\end{proof}
\begin{proposition}\label{prop: choice of T-admissible does not matter}
Let $t_{0},t_{1}$ be $T$-admissible points of the suitable partial adelic point $P_{T}$, then $\hat{R}_{P_{T},t_{0}}\cong \hat{R}_{P_{T},t_{1}}$.
\end{proposition}
\begin{proof}
Suppose $[c][p_{J'}]\in J_{T}$, let $\alpha_{1}:= ev^{t_{0}}(x)$ and $\alpha_{2}:= ev^{t_{1}}(x)$. What we want to show is $\alpha_{1}\in W^{T(t_{0})}$ if and only if $\alpha_{2}\in W^{T(t_{1})}$. We know if $v \in T$ then $\alpha_{1}\in W^{v}$ if and only if $\alpha_{2}\in W^{v}$, so it is sufficient to show that $\alpha_{1}\in W^{u_{i}^{0}}$ if and only if $\alpha_{2} \in W^{u_{i}^{1}}$ for $i\in J$. By Lemma \ref{lemma: selmer group of admissible fibres 2} this is equivalent to showing that $t_{0}$ satisfies condition Lemma \ref{lemma: selmer group of admissible fibres 2} $(2)$ if and only $t_{1}$ does. Let us first assume $i\not\in J'$ then \[
\langle p_{i}(t_{0}),ev^{t_{0}}(x)\rangle_{u_{i}^{0}}=\langle p_{i}(t_{0}),\prod_{i\in J'}p_{j}\left(\frac{-d_{i}}{c_{i}}\right)\rangle_{u_{i}^{0}}.
\] Applying quadratic reciprocity and using the fact that the right entry of the Hilbert symbol in this case is a unit outside of $T$,\[
\langle p_{i}(t_{0}),\prod_{i\in J'}p_{j}\left(\frac{-d_{i}}{c_{i}}\right)\rangle_{u_{i}^{0}}=\sum\limits_{v\in T}\langle p_{i}(t_{0}),\prod_{i\in J'}p_{j}\left(\frac{-d_{i}}{c_{i}}\right)\rangle_{v}=\sum\limits_{v\in T}\langle p_{i}(t_{1}),\prod_{i\in J'}p_{j}\left(\frac{-d_{i}}{c_{i}}\right)\rangle_{v}
\] which is equal to $\langle p_{i}(t_{1}),ev^{t_{1}}(x)\rangle_{u_{i}^{1}}$. A similar calculation shows if $i\in J'$ then the statement holds using the fact that $-cd$ is a unit outside of $T$.
\end{proof}
\begin{remark}
By Proposition \ref{prop: choice of T-admissible does not matter}, we see that the choice of $T$-admissible point does not matter so we can reduce out notation of $
\hat{R}_{P_{T},t_{i}}$ to $\hat{R}_{P_{T}}$.
\end{remark}

\section{Comparing Selmer groups}\label{sec: comparing Selmer groups}
Let $T$ be a finite set of places and $P_{T}$ a suitable partial adelic point. Further, let $t_{0}$ be a $T$-admissible point with associated places $\{u_{i}^{0}\}_{i\in J}$. The goal of this section is to understand what happens to the size of the Selmar and dual Selmar group when one adds to $T$ an additional place $w\not \in S_{0}$. Let $T_{w}:=T\cup \{w\}$ and let $P_{T_{w}}$ be an extension of $P_{T}$ to $T_{w}$. We denote by $t_{1}$ a $T_{w}$-admissible point with respect to $P_{T_{w}}$ with associated places $\{u_{i}^{1}\}_{i\in J}$. We make the following assumption for the rest of Section \ref{sec: comparing Selmer groups}. However, in the proof of Theorem \ref{thm: main thm} in Section \ref{sec: main thm} this assumption will be satisfied by using (H).
\begin{assumption}
There exists $i_{w}\in J$ such that $p_{i_{w}}(t_{w})$ is a uniformizer. 
\end{assumption}

\begin{lemma}\label{Comparing Selmer groups Lemma 1}
For every $i\neq j\in J$ and $i\neq i_{w}$ we have \[
\langle p_{i}(t_{0}),p_{j}(t_{0})\rangle_{u_{i}^{0}}+\langle p_{i}(t_{1}),p_{j}(t_{1})\rangle_{u_{i}^{1}}=0.
\]
\end{lemma}
\begin{proof}
If $i=j\neq i_{w}$ then the statement is trivial. Now assume that $i\neq i_{w}$ and $i\neq j$. As $p_{i}(t_{0})$ (respectively $p_{i}(t_{1})$) is a uniformizer at $u_{i}^{0}$ (respectively $u_{i}^{1}$) we have that $c_{i}t_{0}+d_{i}\equiv 0$ mod $u_{i}^{0}$ and $c_{i}t_{1}+d_{i}\equiv 0$ mod $u_{i}^{1}$ i.e. $t_{0}\equiv \frac{-d_{i}}{c_{i}}$ mod $u_{i}^{0}$ and $t_{1}\equiv \frac{-d_{i}}{c_{i}}$ mod $u_{i}^{1}$. Hence,  \[
\langle p_{i}(t_{0}),p_{j}(t_{0})\rangle_{u_{i}^{0}}+\langle p_{i}(t_{1}),p_{j}(t_{1})\rangle_{u_{i}^{1}}
=\langle p_{i}(t_{0}),-\frac{d_{i}}{c_{i}}c_{j}+d_{j}\rangle_{u_{i}^{0}}+\langle p_{i}(t_{1}),-\frac{d_{i}}{c_{i}}c_{j}+d_{j}\rangle_{u_{i}^{1}}
\] and by quadratic reciprocity \[
\langle p_{i}(t_{0}),-\frac{d_{i}}{c_{i}}c_{j}+d_{j}\rangle_{u_{i}^{0}}+\langle p_{i}(t_{1}),-\frac{d_{i}}{c_{i}}c_{j}+d_{j}\rangle_{u_{i}^{1}}=2\sum_{v\in T}\langle p_{i}(t_{0}),-\frac{d_{i}}{c_{i}}c_{j}+d_{j}\rangle_{v}=0.
\]
\end{proof}

\begin{lemma}\label{lemma: Comparing Selmer groups Lemma 1.5}
Let $x=[c][p_{J'}]\in J^{T}$ and $i\in (J')^{c}$. If $i=i_{w}$ assume that $cp_{J'}(t_{1})$ is a square at $w$ then 
\[
\langle p_{i}(t_{0}),p_{j}(t_{0})\rangle_{u_{i}^{0}}+\langle p_{i}(t_{1}),p_{j}(t_{1})\rangle_{u_{i}^{1}}=0.
\]
\end{lemma}
\begin{proof} Assume $i\neq i_{w}$ then $p_{i}(t_{1})$ is only a non-unit at $u_{i}^{1}$ outside of $T$. Consider
\[
\langle p_{i}(t_{0}),cp_{J'}(t_{0})\rangle_{u_{i}^{0}}+\langle p_{i}(t_{1}),cp_{J'}(t_{1})\rangle_{u_{i}^{1}}
\] by using the bimultiplicative property of Hilbert symbols we have this sum is equal to
\[
\langle p_{i}(t_{0}),[c]\rangle_{u_{i}^{0}}+\langle p_{i}(t_{1}),[c]\rangle_{u_{i}^{1}}+\sum_{j\in J'}[\langle p_{i}(t_{0}),p_{j}(t_{0})\rangle_{u_{i}^{0}}+\langle p_{i}(t_{0}),p_{j}(t_{1})\rangle_{u_{i}^{1}}].
\] As $i\neq i_{w}$ and $c$ is a unit outside of $T$ we have that $\langle p_{i}(t_{0}),[c]\rangle_{u_{i}^{0}}=\langle p_{i}(t_{1}),[c]\rangle_{u_{i}^{1}}$ by quadratic reciprocity. Then by Lemma \ref{Comparing Selmer groups Lemma 1} the statement follows for $i\neq i_{w}$. Now assume $i=i_{w}$. By approximation conditions for all $v\in T$ we have 
$\langle p_{i_{w}}(t_{0}),cp_{J'}(t_{0})\rangle_{v}=\langle p_{i_{w}}(t_{1}),cp_{J'}(t_{1})\rangle_{v}$. 
By quadratic reciprocity 
\[
\langle p_{i_{w}}(t_{1}),cp_{J'}(t_{1})\rangle_{u_{i_{w}}^{1}}=\langle p_{i_{w}}(t_{1}),cp_{J'}(t_{1})\rangle_{w}+\sum\limits_{v\in T}\langle p_{i_{w}}(t_{1}),cp_{J'}(t_{1})\rangle_{v}.
\] 
As $cp_{J'}(t_{1})$ is a square at $w$ we have $\langle p_{i_{w}}(t_{1}),cp_{J'}(t_{1})\rangle_{w}=0$. Hence, $\langle p_{i_{w}}(t_{1}),cp_{J'}(t_{1})\rangle_{u_{i_{w}}^{1}}=\langle p_{i_{w}}(t_{0}),cp_{J'}(t_{0})\rangle_{u_{i_{w}}^{0}}$.
\end{proof}
\subsection{Strictly smaller Selmer groups}
We now what to consider the dual Selmer groups of $\mathcal{T}_{t_{0}}$ and $\mathcal{T}_{t_{1}}$. Let \[
V_{w}:=\HH^{1}_{\et}(\mathbb{Q}_{w},\mathbb{Z}/2\mathbb{Z})\ \ V^{w}:=\HH^{1}_{\et}(\mathbb{Q}_{w},\mathbb{Z}/2\mathbb{Z}).
\]
We have two naturals maps\[
\text{loc}_{w}:J_{T_{w}}\rightarrow V_{w}\ \ \ \text{loc}^{w}:J^{T_{w}}\rightarrow V^{w}.
\] where $\text{loc}_{w}:=\phi_{w}\circ ev_{t_{1}}$ where $\phi_{w}:I_{w}\hookrightarrow V_{w}$ and $\text{loc}^{w}:=\phi^{w}\circ ev^{t_{1}}$ where $\phi^{w}:I^{w}\hookrightarrow V^{w}$.

\begin{definition}
For a finite set of places $H$ containing $S$ and a suitable partial adelic point $P_{H}$ we define the subgroups $R^{0}_{P_{H}}\subseteq R_{P_{H}}$ and $\hat{R}^{0}_{P_{H}}\subseteq \hat{R}_{P_{H}}$ to be the subgroup consisting of elements $[c][p_{J'}]$ such that $i_{w}\not\in J'$.
\end{definition}
\begin{definition}
We denote by $P_{0}\subseteq V_{w}$ (resp $P^{0}\subseteq V^{w}$) the image of $R^{0}_{P_{T}}$ via $\text{loc}_{w}$ (resp image of $\hat{R}^{0}_{P_{T}}$ via $\text{loc}^{w}$) i.e. \[
P_{0}:=\text{loc}_{w}(R^{0}_{P_{T}})\ \ \ P^{0}:=\text{loc}^{w}(\hat{R}^{0}_{P_{T}}).
\]
\end{definition}
\begin{definition}
We denote by $P_{1}\subseteq V_{w}$ (resp $P^{1}\subseteq V^{w}$) the image of $R^{0}_{P_{T_{w}}}$ via $\text{loc}_{w}$ (resp image of $\hat{R}^{0}_{P_{T_{w}}}$ via $\text{loc}^{w}$) i.e. \[
P_{1}:=\text{loc}_{w}(R^{0}_{P_{T_{w}}})\ \ \ P^{1}:=\text{loc}^{w}(\hat{R}^{0}_{P_{T_{w}}}).
\]
\end{definition}
\begin{lemma}\label{lemma: eveluating at new admissible makes no differnce away from w} Let $x=[cp_{J'}]\in J_{T}$ then \begin{enumerate}
\item For $v\in T$ $\ev_{t_{0}}(x)\in W_{v}(t_{0})$ if and only if $\ev_{t_{1}}(x)\in W_{v}(t_{1})$,
\item For $i\neq i_{w}\in J$ we have $\ev_{t_{0}}(x)\in W_{u_{0}^{i}}(t_{0})$ if and only if $\ev_{t_{1}}(x)\in W_{u_{1}^{i}}(t_{1})$.
\end{enumerate}
\end{lemma}
\begin{proof}
$(1)$ follows for approximation conditions in Definition \ref{defn: admissible points}. We saw in Lemma \ref{lemma: selmer conditions} that $ev_{t_{k}}(x)\in W_{u^{k}_{i}}(t_{k})$ for $k=0,1$ if and only if $[cD^{J'}_{i}]_{u_{i}^{k}}=0$. By Lemma \ref{lemma: Comparing Selmer groups Lemma 1.5}\[
\langle p_{i}(t_{0}),cD^{J'}_{i}\rangle_{u_{i}^{0}}+\langle p_{i}(t_{1}),cD^{J'}_{i}\rangle_{u_{i}^{1}}=0.
\] Using the fact that $\text{val}_{u_{i}^{0}}(p_{i}(t_{0}))=\text{val}_{u_{i}^{1}}(p_{i}(t_{1}))=1$ and $cD_{i}^{J'}$ is an $S$-unit, the statement for $(2)$ follows from Lemma \ref{lemma: selmer conditions}.
\end{proof}
\begin{lemma}
The Hilbert pairing of any element in $P_{0}\subseteq V_{w}$ with any element of $P^{1}\subseteq V^{w}$ is trivial.
\end{lemma}
\begin{proof}
Let $x_{1}=[c_{1}][p_{J_{1}}]\in R^{0}_{P_{T}}$ and $x_{2}=[c_{2}][p_{J_{2}}]\in \hat{R}^{0}_{P_{T_{w}}}$ then by definition $i_{w}\not \in J_{1},J_{2}$. Then consider \[
\alpha_{1}:=ev_{t_{1}}(x_{1}) \ \text{and}\ \alpha_{2}:=ev^{t_{1}}(x_{2}).
\] What we want to show is $\langle \alpha_{1},\alpha_{2}\rangle_{w}=0$. As $i_{w}\not\in J_{1},J_{2}$ we have both $\alpha_{1},\alpha_{2}$ are units in $\mathbb{Z}_{u^{1}_{i_{w}}}$. Hence, the Hilbert symbol\[
\langle \alpha_{1},\alpha_{2}\rangle_{u^{1}_{i_{w}}}=0.
\] For $v\in T$ and $v=u^{1}_{i}$ for $i\neq i_{w}$ we have that $\alpha_{1}\in W_{v}(t_{1})$ by Lemma \ref{lemma: eveluating at new admissible makes no differnce away from w} and by definition $\alpha_{2}\in W^{v}(t_{1})$. By the orthogonality of $W^{v}(t_{1}),W_{v}(t_{1})$ we have that \[
\langle \alpha_{1},\alpha_{2}\rangle_{v}=0.
\] As $\alpha_{1},\alpha_{2}$ are units outside of $T_{w}(t_{1})$, we have that for $v\not \in T_{w}(t_{1})$
\[
\langle \alpha_{1},\alpha_{2}\rangle_{v}=0.
\] The only place that is left is $w$, apply quadratic reciprocity

\[
\langle \alpha_{1},\alpha_{2}\rangle_{w}=\sum\limits_{v\in T\cup \{u^{1}_{i}\}_{i\in J}}\langle \alpha_{1},\alpha_{2}\rangle_{v}=0.
\]
\end{proof}
\begin{lemma}
Let $x=[c][p_{J'}]\in \hat{R}^{0}_{P_{T_{w}}}$ be an element such that $\text{loc}^{w}(x)=0$, then $x\in \hat{R}_{P_{T}}^{0}$
\end{lemma}
\begin{proof}
By definition we have $i_{w}\not\in J'$ as $x\in \hat{R}^{0}_{P_{T_{w}}}$. Moreover, as $\text{loc}^{w}(x)=[cp_{J'}(t_{1})]_{w}=0$ in $ \mathbb{Q}^{*}_{w}/(\mathbb{Q}^{*}_{w})^2$, this implies that $c$ is a unit at $w$, hence $c\in \mathbb{Z}_{T}^{*}$. By Selmer conditions we have $ev^{t_{1}}(x)\in W^{v}(t_{1})$ for all $v\in T$ and $ev^{t_{0}}(x)\in W^{v}(t_{0})$ for all $v\in T$. Moreover,\[
\begin{cases}

        \langle p_{i}(t_{1}),ev^{t_{1}}(x)\rangle_{u^{1}_{i}}=0\ &\text{if}\ i\not\in J'\\
        \langle p_{i}(t_{1}),ev^{t_{1}}(x[-d][p_{J}])\rangle_{u^{1}_{i}}=0\ &\text{if}\ i\in J

\end{cases}\] What we need to show is that for $u_{i}^{0}$ \[
\begin{cases}

        \langle p_{i}(t_{0}),ev^{t_{0}}(x)\rangle_{u^{0}_{i}}=0\ &\text{if}\ i\not\in J'\\
        \langle p_{i}(t_{0}),ev^{t_{0}}(x[-d][p_{J}])\rangle_{u^{0}_{i}}=0\ &\text{if}\ i\in J

\end{cases}\]
We first consider the case $i\not\in J'$ (this includes $i=i_{w}$), this follows from Lemma \ref{lemma: Comparing Selmer groups Lemma 1.5}. Now assume $i\in J'$ (which implies $i\neq i_{w}$), then the statement follows from bimultiplicative property of Hilbert symbols, the fact that $-cd$ are units outside of $T$ and Lemma \ref{Comparing Selmer groups Lemma 1}.
\end{proof}
\begin{lemma}[{\cite[Prop 3.8.11,(2)]{H19}}]\label{lemma: strictly smaller selmer}
If $P_{0}$ and $P^{0}$ are non-zero then $P^{1}=\{0\}$ and $\hat{R}^{0}_{P_{T_{w}}}\subsetneq \hat{R}^{0}_{P_{T}}$.
\end{lemma}
\section{Proof of main theorem}\label{sec: main thm}
In this section we will finish the proof of Theorem \ref{thm: main thm}. Recall if $S_{\text{split}}\subseteq S_{0}$ denotes the set of places of $S_{0}$ which split in $K$, then \[
\dim_{\mathbb{F}_{2}}\Sel(\mathcal{T},S)-\dim_{\mathbb{F}_{2}}\Sel(\hat{\mathcal{T}},S)=|S_{\text{split}}|.
\]
\begin{lemma}\label{lemma: if in Gi and satisfies selmer then in GD}
Let $P_{T}$ be a suitable partial adelic point with $t$ an associated admissible point. If $x\in R_{P_{T}}$ lies in $G_{i}$ for some $i$ then $x\in G_{D}$.
\end{lemma}
\begin{proof}
Suppose there exists $x=[c][p_{J'}]\in G_{i}\backslash G_{D}$  i.e. there exists a $i_{x}\neq i$ with associated place $v_{x}$ such that $[aD^{J'}_{i}]_{v_{x}}\neq 0$. As $x$ satisfies the Selmer conditions at $v_{x}$ from Lemma \ref{lemma: selmer conditions} we require $[aD^{J'}_{i}]_{v_{x}}= 0$ which is a contradiction.
\end{proof}
\begin{lemma}\label{lemma: reduce size of dual selmer}
Let $T$ be a finite set of places containing $S$, let $P_{T}$ be a suitable partial adelic point and $t_{0}$ a $T$-admissible point with respect to $P_{T}$ with associated places $\{u_{i}^{0}\}_{i\in J}$. Let $[x]\in \hat{R}_{P_{T}}$ be such that $[x]\not\in \{0,[-d][p_{J}]\}$. Then there exists a place $w$ and an extension of $P_{T}$ to a suitable partial adelic point $P_{T_{w}}$, where $T_{w}:=T\cup\{w\}$ such that $\hat{R}_{P_{T_{w}}}\subsetneq \hat{R}_{P_{T}}$.
\end{lemma}
\begin{proof} Let $x_{0}:=x=[c_{0}][p_{J_{0}}]$. By assumption there exists a place $v\in S_{0}$ such that $v\in S_{\text{split}}$  we know the difference $\dim_{\mathbb{F}_{2}}R_{P_{T}}-\dim_{\mathbb{F}_{2}}\hat{R}_{P_{T}}\geq 1$ and by assumption $\dim_{\mathbb{F}_{2}}\hat{R}_{P_{T}}\geq 2$.  Hence, \[
\dim_{\mathbb{F}_{2}}R_{P_{T}}>\dim_{\mathbb{F}_{2}}\hat{R}_{P_{T}}\geq 2.
\] and we can deduce there exists $x_{1}=[c_{1}][p_{J'}]\in R_{P_{T}}$ such that $x_{1}\not\in \langle [a][p_{A}],[d][p_{J}]\rangle \subseteq G$. By Condition (D) there exists $i_{x}\in J$ such that $c_{1}\hat{D}^{J_{1}}_{i_{x}}\not\in \{1,[aD_{i_{x}}^{A}]$. We can assume $i_{x}\not\in J_{0}\cup J_{1}$ as we can replace $x_{0}$ with $x_{0}[-dp_{J}]$ and $x_{1}$ with $x_{1}[dp_{J}]$. If $x_{1}\in R_{P_{T}}$ and $x_{1}\in G_{i_{x}}$ by Lemma \ref{lemma: if in Gi and satisfies selmer then in GD} then $x_{1}\in G_{D}$ which would contradict Condition (D), hence, $x_{1}\not\in G_{i_{x}}$. Applying Chebotarev's density theorem, there exists a place $w$ such $aD_{i_{x}}^{A}$ is a square at $w$ but $c_{0}D^{J_{0}}_{i_{x}}$ and $c_{1}\hat{D}^{J_{1}}_{i_{x}}$ are non-squares. We can choose $t_{w}\in \mathbb{Z}_{w}$ such that $p_{i}(t_{w})$ is a uniformizer at $w$, using (H). By Lemma \ref{lemma: locally soluble fibre} the fibre $\mathcal{U}_{t_{w}}(\mathbb{Z}_{w})\neq \emptyset$ and we may extend our suitable partial adelic point $P_{T}$ over $T$ to a suitable partial adelic point $P_{T_{w}}$ over $T_{w}$, where $P_{w}$ is a point above $t_{w}$. By Proposition \ref{prop: can always find T-admissible point} we can find a $T_{w}$-admissible point $t_{w}$ with respect to $P_{T_{w}}$. We denote by $\{u_{i}^{1}\}_{i\in J}$ the associated places of $t_{1}$ appearing from applying (H). We have that $x_{0}\in \hat{R}_{P_{T}}^{0}$ and $x_{1}\in R_{P_{T}}^{0}$ as $i_{x}\not\in J_{0}\cup J_{1}$. By our choice of $w$ we have that $\text{loc}^{w}(x_{0}),\text{loc}_{w}(x_{1})\neq 0$. Hence, $P^{0},P_{0}\neq \{0\}$. By Lemma \ref{lemma: strictly smaller selmer} we have that 
$\hat{R}^{0}_{P_{T_{w}}}\subsetneq \hat{R}^{0}_{P_{T}}$. As 
\[
\hat{R}_{P_{T}}=\hat{R}_{P_{T}}^{0}\oplus \mathbb{F}_{2}\langle [-d][p_{J}]\rangle,\ \ \ \hat{R}_{P_{T_{w}}}=\hat{R}_{P_{T_{w}}}^{0}\oplus \mathbb{F}_{2}\langle [-d][p_{J}]\rangle 
\]we have $\hat{R}_{P_{T_{w}}}\subsetneq \hat{R}_{P_{T}}$.
\end{proof}
\begin{proof}[\textbf{Proof of Theorem \ref{thm: main thm}}]
We can always find a suitable partial adelic point $P_{S}$ by Proposition \ref{prop: Can always find suitable partial adelic point}. Invoking Proposition \ref{prop: can always find T-admissible point} we can find an $S$-admissible point $t_{0}$ associated to $P_{S}$. Then we can apply Lemma \ref{lemma: reduce size of dual selmer} recursively to find a finite set of places $S\subseteq T$ such that we can extend $P_{S}$ to $P_{T}$ and for every $T$-admissible point $t_{1}$ associated to $P_{T}$ has the dual Selmar group $\Sel(\hat{\mathcal{T}}_{t_{1}},T_{0}(t_{1}))$ is generated by $[-dp_{J}(t_{1})]$, as the choice of $T$-admissible point does not matter. Applying Proposition \ref{prop: harpaz result for hasse principle and size of selmer group} we can deduce $\mathcal{U}_{t_{1}}(\mathbb{Z}_{S_{0}})\neq \emptyset$, hence $\mathcal{U}(\mathbb{Z}_{S_{0}})\neq \emptyset$.
\end{proof}

\bibliographystyle{amsalpha}{}
\bibliography{bibliography_project4/referencesproject4}
\end{document}